\documentclass[12pt]{article}
\usepackage{amsmath, amscd, amssymb, latexsym, epsfig, color, amsthm, mathtools, tikz-cd, array, adjustbox, bbm, tikz}
\numberwithin{equation}{section}

\newtheorem{theorem}{Theorem}[section]
\newtheorem{proposition}[theorem]{Proposition}
\newtheorem{question}[theorem]{Question}
\newtheorem{corollary}[theorem]{Corollary}
\newtheorem{conjecture}[theorem]{Conjecture}

\theoremstyle{definition}
\newtheorem{construction}[theorem]{Construction}

\newtheorem{example}[theorem]{Example}
\newtheorem{remark}[theorem]{Remark}

\DeclareMathOperator\lk{\mathrm{lk}}

\DeclareMathOperator\cost{\mathrm{cost}}
\DeclareMathOperator{\Hilb}{\mathrm{Hilb}}
\DeclareMathOperator{\Coker}{\mathrm{Coker}}

\newcommand{\field}{\mathbbm{k}}

\newcommand{\ZZ}{{\mathbb Z}}

\newcommand{\im}{\operatorname{Im}}

\newcommand{\mideal}{\ensuremath{\mathfrak{m}}}

\newcommand{\Hom}{\ensuremath{\mathrm{Hom}}\hspace{1pt}}
\newcommand{\Ext}{\ensuremath{\mathrm{Ext}}\hspace{1pt}}
\newcommand{\Ker}{\ensuremath{\mathrm{Ker}}\hspace{1pt}}

\title{Almost Buchsbaumness of some rings arising from complexes with isolated singularities}
\author{Connor Sawaske\\
\small Department of Mathematics\\[-0.8ex]
\small University of Washington\\[-0.8ex]
\small Seattle, WA 98195-4350, USA\\[-0.8ex]
\small \texttt{sawaske@math.washington.edu}
}

\begin{document}
%%%%%%%%%%%%%%%%%%%%%%%%%%%%%%%%%%%%%%%%%%%%
%%%%%%%%%%%%%%%%%%%%%%%%%%%%%%%%%%%%%%%%%%%%
\maketitle

\begin{abstract}
We study properties of the Stanley--Reisner rings of simplicial complexes with isolated singularities modulo two generic linear forms. Miller, Novik, and Swartz proved that if a complex has homologically isolated singularities, then its Stanley--Reisner ring modulo one generic linear form is Buchsbaum. Here we examine the case of non-homologically isolated singularities, providing many examples in which the Stanley--Reisner ring modulo two generic linear forms is a quasi-Buchsbaum but not Buchsbaum ring.
\end{abstract}

%%%%%%%%%%%%%%%%%%%%%%%
%Introduction
%%%%%%%%%%%%%%%%%%%%%%%

\section{Introduction}

Many combinatorial, algebraic, and topological statements about polytopes and triangulations of spheres or manifolds have been proven through the study of their Stanley--Reisner rings. These rings are well-understood, and the translation of their algebraic properties into combinatorial and topological invariants has a storied and celebrated past. The usefulness of this approach is made apparent by its presence in decades of continued progressive research (excellent surveys may be found in \cite{St-96} and \cite{KN-faces}).

In contrast, the main objects considered in this paper are simplicial complexes with isolated singularities. A simplicial complex $\Delta$ has isolated singularities if it is pure and the link in $\Delta$ of every face of dimension at least $1$ is Cohen-Macaulay (more precise definitions will be provided later). Common examples are provided by triangulations of a pinched torus, the suspension of a manifold, or more generally by pseudomanifolds that fail to be manifolds at finitely many points (see \cite{GoreskyMacPherson} for an in-depth discussion). The gap between pseudomanifolds and manifolds is well understood from a topological viewpoint, but there are powerful tools available to the Stanley--Reisner rings of triangulations of manifolds that presently lack any meaningful extension to the world of pseudomanifolds. For instance, unlike results due to Stanley (\cite{St-UBC}) and Schenzel (\cite{Schenzel}), we do not know the Hilbert series of a generic Artinian reduction of the Stanley--Reisner ring of such a complex, even when considering a triangulation of the suspension of a manifold that is not a homology sphere.

The central obstruction in extending the pre-existing knowledge of triangulations of manifolds to the singular case is the Stanley--Reisner ring failing to be Buchsbaum. Miller, Novik, and Swartz were able to circumvent this roadblock in the case that the singularities in question are homologically isolated. By showing that a certain quotient of the associated Stanley--Reisner ring is in fact Buchsbaum, they established some enumerative theorems related to $f$- and $h$-vectors in \cite{MNS-sing}. Novik and Swartz were then able to calculate the Hilbert series of a generic Artinian reduction of the Stanley--Reisner ring as well as prove singular analogs of the Dehn-Sommerville relations in \cite{NS-sing}.

These established results were the inspiration for this paper. However, the ultimate purpose here is twofold; for one, we intend for the algebraic implications of isolated singularities to become as equally well-understood as the topological ones. To this end, we will investigate with some precision how the topological properties of singular vertices translate to the algebraic setting of Stanley--Reisner rings. This, in particular, leads to a notion of {\em generically isolated} singularities, defined in Section \ref{sect:preliminaries}. This notion plays a crucial role in our main results described below. We will exhibit similarities and differences between the singular and non-singular cases, examine some special examples, and provide alternate interpretations of some classical results. In doing so, we will see that these rings have interesting properties that are worth studying in their own right; these properties provide the second purpose for this paper. In particular, we present new findings demonstrating that some quotients of Stanley-Reisner rings of simplicial complexes with isolated singularities are very near to being Buchsbaum, and we characterize when and to what degree this occurs.

\begin{theorem}\label{annihilateThm}Let $\Delta$ be a connected simplicial complex with isolated singularities on vertex set $V$, let $\field$ be an infinite field, and denote by $A$ the polynomial ring $\field[x_v: v\in V]$ and by $\mideal$ the irrelevant ideal of $A$. If $\theta_1, \theta_2$ is a generic regular sequence for the Stanley-Reisner ring $\field[\Delta]$, then the local cohomology module $H_\mideal^i(\field[\Delta]/(\theta_1, \theta_2)\field[\Delta])$ satisfies
	\[
	\mideal\cdot H_\mideal^i(\field[\Delta]/(\theta_1, \theta_2)\field[\Delta])=0
	\]
	for all $i$ if and only if the singularities of $\Delta$ are generically isolated.
\end{theorem}

\begin{theorem}\label{surjectivityThm}In the setting of Theorem \ref{annihilateThm}, the canonical maps of graded modules
	\[
	\varphi^i:\Ext_A^i(\field, \field[\Delta]/(\theta_1, \theta_2)\field[\Delta])\to H_\mideal^i(\field[\Delta]/(\theta_1, \theta_2)\field[\Delta])
	\]
	are surjective in all degrees except (possibly) $0$ if and only if the singularities of $\Delta$ are generically isolated.
\end{theorem}

The condition $\mideal\cdot H_\mideal^i(\field[\Delta]/(\theta_1, \theta_2)\field[\Delta])=0$ in Theorem \ref{annihilateThm} establishes a strong necessary condition for a ring to be Buchsbaum (see \cite[Corollary 1.5]{GS-Index}), known as quasi-Buchsbaumness. These rings were first introduced by Goto and Suzuki (\cite{GS-Index}). As with Buchsbaum rings, the properties and various characterizations of quasi-Buchsbaum rings have long been of some interest (see, e.g., \cite{S-Quasi} and \cite{Y-Quasi}). The usefulness of the property is evidenced, for example, by its equivalence to Buchsbaumness in some special cases (\cite[Corollary 3.6]{StVo}).

The maps $\varphi^i$ being surjective in Theorem \ref{surjectivityThm} is a complement to the quasi-Buchsbaum property in that it is incredibly near to one characterization of Buchsbaumness (see Theorem \ref{surjective}). As we will see in Corollary \ref{SuspensionsAlmostBuchs}, if $\Gamma$ is a Buchsbaum (but not Cohen-Macaulay) complex and $\Delta$ is an arbitrary triangulation of the geometric realization of the suspension of $\Gamma$, then $\field[\Delta]/(\theta_1, \theta_2)\field[\Delta]$ is never Buchsbaum. Though Vogel (\cite{V-Quasi}) and Goto (\cite{G-Quasi}) provided initial examples of quasi-Buchsbaum but not Buchsbaum rings, here we exhibit an infinite family of quasi-Buchsbaum rings of arbitrary dimensions and varying depths which fail in a geometrically tangible way to be fully Buchsbaum in only the slightest sense.

The structure of the paper is as follows. In Section \ref{sect:preliminaries} we provide definitions and foundational results, allowing for some initial computations in Section \ref{calculations}. We prove our main results in Section \ref{results}, and in Section 5 we will use some properties of quasi-Buchsbaum rings to calculate the Hilbert series of a certain Artinian reduction of the Stanley--Reisner ring of a complex with isolated singularities. We will close with comments and open problems in Section \ref{comments}.

%%%%%%%%%%%%%%%%%%%%%%%%%%%
%Preliminaries
%%%%%%%%%%%%%%%%%%%%%%%%%%%

\section{Preliminaries} \label{sect:preliminaries}

This paper has been largely influenced by the works of Miller, Novik, and Swartz in \cite{MNS-sing} and Novik and Swartz in \cite{NS-sing}. In order to retain consistency with these references, much of their notation will be adopted for our uses as well.

\subsection{Combinatorics and topology}

%%%%%%%%%%%%
%Combinatorics and Geometry
%%%%%%%%%%%%%

A \textbf{simplicial complex} $\Delta$ with vertex set $V$ is a collection of subsets of $V$ that is closed under inclusion. The elements of $\Delta$ are called \textbf{faces}, and the \textbf{dimension} of a face $F$ is $\dim F:=|F|-1$. The $0$-dimensional faces in $\Delta$ are called \textbf{vertices}, and the maximal faces under inclusion are called \textbf{facets}. We say that $\Delta$ is \textbf{pure} if all facets have the same dimension. The dimension of the complex $\Delta$ is $\dim \Delta:=\max\{\dim F : F\in \Delta\}$. For the remainder of this paper, unless stated otherwise we will assume that a simplicial complex $\Delta$ is pure of dimension $d-1$ with vertex set $V$.

The \textbf{link} of a face $F$ is the subcomplex of $\Delta$ defined by
\[
\lk_\Delta F =\{G\in \Delta: F\cup G\in \Delta, F\cap G = \emptyset\},
\]
and the \textbf{contrastar} of a face $F$ is defined by
\[
\cost_\Delta F = \{G\in \Delta: F\not\subset G\}.
\]
In the case that $F=\{v\}$ is a vertex, we write $\lk_\Delta v $ and $\cost_\Delta v$ instead of $\lk_\Delta\{v\}$ and $\cost_\Delta \{v\}$, respectively. If $W\subset V$, then the \textbf{induced subcomplex} $\Delta_W$ is the simplicial complex $\{F\in \Delta: F\subset W\}$. For $0\le i\le d-1$, the complex $\Delta^{(i)}:=\{F\in \Delta: |F|\le i+1\}$ is the \textbf{$i$-skeleton} of $\Delta$.

Given a field $\field$, denote by $H_i(\Delta)$ and $H^i(\Delta)$ the $i$th (simplicial) homology and cohomology groups of $\Delta$ computed over $\field$, respectively (definitions and further resources may be found in \cite{Hatcher}). If $F$ is a face of $\Delta$, we denote by $H^i_F(\Delta)$ the relative (simplicial) cohomology group $H^i(\Delta, \cost_\Delta F)$ (in line with other conventions, we write $H^i_v(\Delta)$ for $H^i_{\{v\}}(\Delta)$). It will often be helpful to identify $H_F^i(\Delta)$ with $\tilde{H}^{i-|F|}(\lk_\Delta F)$ (see, e.g., \cite[Section 1.3]{Grabe}); in particular, note that $H^i_\emptyset(\Delta) = \tilde{H}^i(\Delta)$, the reduced cohomology group of $\Delta$. Finally, let
\[\iota_F^i:H^i_F(\Delta)\to H^i_\emptyset(\Delta)\]
be the map induced by the inclusion $(\Delta, \emptyset)\to (\Delta, \cost F)$.

 If $H_\emptyset^0(\Delta)=0$, then we call $\Delta$ \textbf{connected}. We say that a face $F$ of $\Delta$ is \textbf{singular} if $H_F^i(\Delta)\not=0$ for some $i<d-1$. Conversely, if $H_F^i(\Delta)=0$ for all $i<d-1$, we call $F$ a \textbf{nonsingular} face. We call $\Delta$ \textbf{Cohen-Macaulay} over $\field$ if every face of $\Delta$ (including $\emptyset$) is nonsingular, and we call $\Delta$ \textbf{Buchsbaum} (over $\field$) if it is pure and every face aside from $\emptyset$ is nonsingular.

If $\Delta$ contains a singular face, then we define the \textbf{singularity dimension} of $\Delta$ to be $\max\{\dim F: F\in\Delta$ and $F$ is singular$\}$. If the singularity dimension of $\Delta$ is $0$, then we say that $\Delta$ has \textbf{isolated singularities}. Such complexes will be our main objects of study. As a special case, if the images of the maps $\iota_v^i:H_v^i(\Delta)\to H_\emptyset^i(\Delta)$ are linearly independent as vector subspaces of $H_\emptyset^i(\Delta)$, then we call the singularities of $\Delta$ \textbf{homologically isolated}. Equivalently, the singularities of $\Delta$ are homologically isolated if the kernel of the sum of maps
\[
\left(\sum_{v\in V} \iota_v^i\right):\bigoplus_{v\in V}H_v^i(\Delta)\to H_\emptyset^i(\Delta)
\]
decomposes as the direct sum
\[
\bigoplus_{v\in V}\left(\Ker \iota_v^i:H_v^i(\Delta)\to H_\emptyset^i(\Delta)\right).
\]
Lastly, we call the singularities of $\Delta$ \textbf{generically isolated} if for sufficiently generic choices of coefficients $\{\alpha_v:v\in V\}$ and $\{\gamma_v:v\in V\}$, the two maps $\theta_\alpha$ and $\theta_\gamma$ defined by
\[
\left(\sum_{v\in V} \alpha_v\iota_v^i\right):\bigoplus_{v\in V}H_v^i(\Delta)\to H_\emptyset^i(\Delta)\hspace{15pt}
\text{ and }\hspace{15pt}
\left(\sum_{v\in V} \gamma_v\iota_v^i\right):\bigoplus_{v\in V}H_v^i(\Delta)\to H_\emptyset^i(\Delta),
\]
respectively, satisfy
\[
(\Ker \theta_\alpha) \cap (\Ker \theta_\gamma) =\bigoplus_{v\in V}\left(\Ker \iota_v^i:H_v^i(\Delta)\to H_\emptyset^i(\Delta)\right)
\]
for all $i$.

%%%%%%%%%%%%%%%%%%%
%Algebra
%%%%%%%%%%%%%%

\subsection{The connection to algebra}\label{Algebra}
For the remained of the paper, let $\field$ be a fixed infinite field. Define $A:=\field[x_v:v\in V]$ and let $\mideal=(x_v:v\in V)$ be the irrelevant ideal of $A$. If $F\subset V$, let
\[
x_F=\prod_{v\in F}x_v
\]
and define the \textbf{Stanley-Reisner ideal} $I_\Delta$ by
\[
I_\Delta = (x_F: F\not\in \Delta).
\]
The ring $\field[\Delta]:=A/I_\Delta$ is the \textbf{Stanley-Reisner ring} of $\Delta$. We will usually consider $\field[\Delta]$ as an $A$-module that is graded with respect to $\ZZ$ by degree.

Given any $\ZZ$-graded $A$-module $M$ of Krull dimension $d$, we denote by $M_j$ the collection of homogeneous elements of $M$ of degree $j$. A sequence $\Theta=(\theta_1, \theta_2, \ldots, \theta_d)$ of linear forms in $A$ is called a \textbf{homogeneous system of parameters} (or \textbf{h.s.o.p.}) for $M$ if each $\theta_i$ is a homogeneous element of $A$ and $M/\Theta M$ is a finite-dimensional vector space over $\field$. In the case that each $\theta_i$ is linear, we call $\Theta$ a \textbf{linear system of parameters} (or \textbf{l.s.o.p.}) for $M$. If
\[(\theta_1, \ldots, \theta_{i-1})M:_M\theta_i = (\theta_1, \ldots, \theta_{i-1})M:_M\mideal
\] 
for $1\le i\le d$ for any choice of $\Theta$, then we call $M$ \textbf{Buchsbaum}. The reasoning behind the geometric definition of a Buchsbaum simplicial complex is made apparent in the following theorem due to Schenzel (\cite{Schenzel}):
\begin{theorem}A pure simplicial complex $\Delta$ is Buchsbaum over $\field$ if and only if $\field[\Delta]$ is a Buchsbaum $A$-module.
\end{theorem}
Two collections of objects that are most vital to our results are the $A$-modules \linebreak $\Ext_A^i(\field, \field[\Delta])$ and $H_\mideal^i(\field[\Delta])$. An excellent resource for their construction and basic properties is \cite{24Hours}. In the case of Stanley-Reisner rings, the $\ZZ$-graded structure of $\Ext_A^i(\field, \field[\Delta])$ is provided by \cite[Theorem 1]{characterizations}:

\begin{theorem}[Miyazaki]\label{ExtStructure}Let $\Delta$ be a simplicial complex. Then as vector spaces over $\field$,
	\[
	\Ext_A^i(\field, \field[\Delta])_j\cong \left\{\begin{array}{*3{>{\displaystyle}l}} 0 & \,\,\, & j<-1\text{ or }j>0 \\ H_\emptyset^{i-1}(\Delta) & & j=0 \\ \bigoplus_{v\in V}H_\emptyset^{i-2}(\cost_\Delta v) & & j=-1.\end{array} \right.
	\]
\end{theorem}

 The local cohomology modules $H^i_\mideal(\field[\Delta])$ may be computed as the direct limit of the $\Ext_A^i(A/\mideal^l, \field[\Delta])$ modules (see \cite[Corollary 1]{characterizations}). However, we will also need a thorough understanding of the $A$-module structure of the local cohomology modules of $\field[\Delta]$. The necessary details are provided by \cite[Theorem 2]{Grabe}:

\begin{theorem}[Gr\"abe]\label{Grabe}Let $\Delta$ be a $(d-1)$-dimensional simplicial complex with isolated singularities and let $0\le i\le d$. Then
\[
H^i_\mideal(\field[\Delta])_j\cong \left\{\begin{array}{*3{>{\displaystyle}l}} 0 & \,\,\, & j>0 \\ H_\emptyset^{i-1}(\Delta) & & j=0 \\ \bigoplus_{v\in V}H_v^{i-1}(\Delta) & & j<0\end{array} \right.
\]
as $\field$-vector spaces. If $\alpha_v\in H^{i-1}_v(\Delta)$, then the $A$-module structure on $H^i_\mideal(\field[\Delta])$ is given by
\[
\begin{array}{ccc}\cdot x_u: H^i_\mideal(\field[\Delta])_{j-1} & \xrightarrow{\hspace{30pt}} & H^i_\mideal(\field[\Delta])_j\\
\alpha_v &  \xmapsto{\hspace{30pt}} & \left\{ \begin{array}{cc}\alpha_v & j<0\text{ and }u=v \\ \iota_v^{i-1}(\alpha_v)  & j=0\\ 0 & \text{otherwise.}\end{array}\right.
\end{array}
\]
\end{theorem}

It is immediate from Theorem \ref{Grabe} that a pure simplicial complex $\Delta$ is Buchsbaum if and only if it is pure and $H_\mideal^i(\field[\Delta])$ is concentrated in degree $0$ for all $i\not=d$. We note at this point that an $A$-module $M$ is called \textbf{quasi-Buchsbaum} if $\mideal\cdot H_\mideal^i(M)=0$ for all $i$. Evidently, the above theorem implies that $\field[\Delta]$ is Buchsbaum if and only if it is quasi-Buchsbaum; this is certainly not the case for general modules, as we will see!

Elements and homomorphisms related to $H_\mideal^i(\field[\Delta])$ will usually be represented and referenced according to the topological identifications above; these identifications will also be expanded as we progress. As a motivating example, consider the $\field$-vector space
\[
K^i := \bigoplus_{v\in V}\left(\Ker \iota^{i-1}_v:H^{i-1}_v(\Delta)\to H^{i-1}_\emptyset(\Delta)\right).
\]
This space is central to the definition of homological isolation of singularities. Sometimes, it will be easier to identify $K^i$ with its counterpart in local cohomology. If we denote by $H_\mideal^i(\field[\Delta])_v$ the direct summand of $H_\mideal^i(\field[\Delta])_{-1}$ corresponding to $H_v^{i-1}(\Delta)$ in Theorem \ref{Grabe}, then $K^i$ is identified with the submodule 
\[
\bigoplus_{v\in V}\Ker \left(\cdot x_v: H_\mideal^i(\field[\Delta])_v\to H_\mideal^i(\field[\Delta])\right)
\]
of $H_\mideal^i(\field[\Delta])$. In light of how intertwined these two objects are, we will use the notation $K^i$ interchangeable between the two settings. In local cohomology, the equality
\[
K^i=\Ker\left(\sum_{v\in V} \cdot x_v :\bigoplus_{v\in V}H_\mideal^i(\field[\Delta])_v\to H_\mideal^i(\field[\Delta])_0\right)
\]
is equivalent to the singularities of $\Delta$ being homologically isolated. In the same way, if $\theta_1$ and $\theta_2$ are generic linear forms, then the generic isolation of the singularities of $\Delta$ is equivalent to the equality
\[
\Ker\left(\cdot\theta_1:\bigoplus_{v\in V}H_\mideal^i(\field[\Delta])_v\to H_\mideal^i(\field[\Delta])_0\right)\cap \Ker\left(\cdot\theta_2:\bigoplus_{v\in V}H_\mideal^i(\field[\Delta])_v\to H_\mideal^i(\field[\Delta])_0\right)=K^i
\]
holding for all $i$.

\begin{remark}In their original statements, the above structure theorems are written with respect to a $\ZZ^{|V|}$-grading and are proved by examining the chain complex $\Hom_A(\mathcal{K}^l_\cdot, \field[\Delta])$, where $\mathcal{K}^l_\cdot$ is the Koszul complex of $A$ with respect to $\mideal^l$. When coarsening to a $\ZZ$-grading, this chain complex becomes much larger. However, an argument similar to Reisner's original proof that $H^i_\mideal(k[\Delta])_0\cong H_\emptyset^{i-1}(\Delta)$ (see \cite{Reisner}, pp. 41-42) shows that the only potentially non-acyclic components of $\Hom_A(\mathcal{K}^l_\cdot, M)$ under a $\ZZ$-grading are those also appearing in the $\ZZ^{|V|}$-graded complex.
\end{remark}

\section{Auxiliary calculations}\label{calculations}Unless stated otherwise, we will always assume that $\Delta$ is a connected $(d-1)$-dimensional simplicial complex with isolated singularities. Let $V$ be the vertex set of $\Delta$ and set $A:=\field[x_v: v\in V]$. We will always consider $R:=\field[\Delta]$ as an $A$-module. If $\theta_1, \ldots, \theta_d$ is a homogeneous system of parameters for $\Delta$, we denote $R^i := \field[\Delta]/(\theta_1, \ldots, \theta_i)\field[\Delta]$.

Since $\Delta$ is connected, the $1$-skeleton $\Delta^{(1)}$ of $\Delta$ is Cohen-Macaulay and the depth of $R$ is at least $2$ (see \cite[Corollary 2.6]{HibiDepth}). Hence, there exists a homogeneous system of parameters $\theta_1, \ldots, \theta_d$ for $\Delta$ in which $\theta_1, \theta_2$ are linear and form the beginning of a regular sequence for $R$, i.e., $\theta_1$ is a non-zero-divisor on $R$ and $\theta_2$ is a non-zero-divisor on $R^1$. Generically, we may assume that $\theta_1$ and $\theta_2$ both have non-zero coefficients on all $x_v$'s. Unless stated otherwise, we will always work with such a system of parameters for $\Delta$.

Our results primarily depend upon an understanding of the $A$-modules $H_\mideal^i(R^j)$. We begin by computing their dimensions over $\field$ when $j=1$ or $2$ and discussing some connections to the topology of $\Delta$.

\subsection{Local cohomology}\label{cohomology1}
Consider the exact sequence of $A$-modules
\[
0\to R\xrightarrow{\cdot\theta_1} R\xrightarrow{\pi} R^1 \to 0.
\]
By looking at graded pieces of this sequence, there are exact sequences of vector spaces over $\field$ of the form
\begin{equation}\label{S1}
0\to R_{j-1}\xrightarrow{\cdot\theta_1}R_j\xrightarrow{\pi} R^1_j\to 0.
\end{equation}
These sequences induce the following long exact sequence in local cohomology, where $\theta_1^{i, j}:H_\mideal^i(R)_{j-1}\to H_\mideal^i(R)_j$ is the map induced by multiplication, $\pi$ is the map induced by the projection $R\to R^1$, and $\delta$ is the connecting homomorphism:
\begin{equation}\label{LES1}
H^i_\mideal(R)_{j-1}\xrightarrow{\theta_1^{i, j}} H^i_\mideal(R)_j\xrightarrow{\pi} H^i_\mideal(R^1)_j\xrightarrow{\delta}H^{i+1}_\mideal(R)_{j-1}\xrightarrow{\theta_1^{i+1, j}} H^{i+1}_\mideal(R)_j.
\end{equation}
In light of Theorem \ref{Grabe}, we make the following conclusions. When $j>1$, all terms are zero. When $j=1$, $\delta$ is an isomorphism. When $j\le-1$, each $\theta_1^{i, j}$ is an isomorphism (all coefficients of $\theta_1$ are non-zero). When $j=0$, we obtain the short exact sequence
\begin{equation}\label{kercokersequence1}
0\to \Coker\theta_1^{i, 0}\to H_\mideal^i(R^1)_0\to \Ker\theta_1^{i+1, 0}\to 0.
\end{equation}
Hence, as $\field$-vector spaces,
\begin{equation}\label{S1cohomology}
H^i_\mideal(R^1)_j \cong \left\{\begin{array}{cc} H_\emptyset^i(\Delta) & j=1 \\ \Coker \theta_1^{i, 0}\oplus \Ker \theta_1^{i+1, 0} & j= 0 \\ 0 & \text{otherwise.} \end{array}\right.
\end{equation}

It will be useful to keep in mind that $\Coker \theta_1^{i, 0}$ is identified with a quotient of $H^{i-1}_\emptyset(\Delta)$ and that $\Ker\theta_1^{i+1, 0}$ is identified with a submodule of $\oplus_{v\in V}H_v^i(\Delta)$. Although the short exact sequence (\ref{kercokersequence1}) above is not necessarily split, we will leverage the ``geometric'' $A$-module structures of $\Coker \theta_1^{i, 0}$ and $\Ker\theta_1^{i+1, 0}$ along with (\ref{kercokersequence1}) to further analyze $H_\mideal^i(R^1)$ in Section \ref{results}.

Now we repeat this argument; consider the short exact sequence
\begin{equation}\label{S2}
0 \to R^1_{j-1}\xrightarrow{\cdot\theta_2} R^1_j\xrightarrow{\pi} R^2_j\to 0
\end{equation}
of vector spaces over $\field$, giving rise to the long exact sequence
\begin{equation}\label{LES1}
H^i_\mideal(R^1)_{j-1}\xrightarrow{\theta_2^{i, j}}H^i_\mideal(R^1)_j\xrightarrow{\pi} H^i_\mideal(R^2)_j\xrightarrow{\delta}H^{i+1}_\mideal(R^1)_{j-1}\xrightarrow{\theta_2^{i+1, j}}H^{i+1}_\mideal(R^1)_j.
\end{equation}

As in the previous computation, all terms are zero when $j<0$ or $j>2$, $\pi$ is an isomorphism when $j=0$, and $\delta$ is an isomorphism when $j=2$. When $j=1$, we have the exact sequence
\begin{equation}\label{kercokersequence2}
0\to \Coker\theta_2^{i, 1}\to H_\mideal^i(R^2)_1\to \Ker\theta_2^{i+1, 1}\to 0.
\end{equation}
Hence, as vector spaces over $\field$,
\begin{equation}\label{S2cohomology}
H^i_\mideal(R^2)_j \cong \left\{\begin{array}{*2{>{\displaystyle}c}} H_\emptyset^{i+1}(\Delta) & j=2 \\\Coker \theta_2^{i, 1}\oplus \Ker \theta_2^{i+1, 1} & j=1 \\\Coker \theta_1^{i, 0}\oplus \Ker \theta_1^{i+1, 0} & j= 0 \\ 0 & \text{otherwise.} \end{array}\right.
\end{equation}

\subsection{Local cohomology: suspensions}\label{cohomology2}
We will now briefly consider the special case of suspensions. Suppose $\Delta$ is an arbitrary triangulation of the suspension of a $(d-2)$-dimensional manifold that is not Cohen-Macaulay, with suspension points $a$ and $b$ (so that $a$ and $b$ are isolated singularities of $\Delta$). In this context, the maps $\iota_a^i$ and $\iota_b^i$ from Section \ref{sect:preliminaries} are isomorphisms. If $g^i$ is a generator for $H_\emptyset^i(\Delta)$, denote $g_a^i:=(\iota_a^i)^{-1}(g)\in H_a^i(\Delta)$ and $g_b^i:=(\iota_b^i)^{-1}(g)\in H_b^i(\Delta)$. As usual, we will consider these generators interchangeable with their corresponding elements in $H_\mideal^{i+1}(R)$.

Examining the sequence in (\ref{kercokersequence1}) for this special case, suppose $\theta_1=\sum_{v\in V}x_v$ and $\theta_2=\sum_{v\in V}c_vx_v$ with $c_a\not= c_b$ and $c_a, c_b\not=0$. Given $g^{i-1}$ a generator of $H_\emptyset^{i-1}(\Delta)$, the map $\theta_1^{i, 0}$ acts via $\theta_1^{i, 0}(g_a^{i-1})=\theta_1^{i, 0}(g_b^{i-1})=g^{i-1}$. Hence, $\theta_1^{i, 0}$ is a surjection whose kernel is generated as a direct sum by elements of the form $(g_a^{i-1}-g_b^{i-1})$. In particular, $H^i_\mideal(R^1)_0\cong \Ker(\theta_1^{i+1, 0})\cong H_\emptyset^i(\Delta)$. In summary,
\[
H^i_\mideal(R^1)_j \cong \left\{\begin{array}{cc} H_\emptyset^i(\Delta) & j=1 \\ H_\emptyset^i(\Delta) & j= 0 \\ 0 & \text{otherwise,} \end{array}\right.
\]
under the aforementioned isomorphisms. 

Now repeat the process above using the sequence in (\ref{kercokersequence2}). If $g^{i}$ is a generator of $H_\emptyset^i(\Delta)$, then $\theta_2^{i+1, 0}$ acts on $H^{i+1}_\mideal(R)_{-1}$ via $\theta_2^{i+1, 0}(g_a^i)=c_ag^i$ and $\theta_2^{i+1, 0}(g_b^i)=c_bg^i$. In particular, identifying $H^{i+1}_\mideal(R^1)_0$ with the subspace $\Ker(\theta_1^{i+1, 0})$ of $H^{i+1}_\mideal(R)_1$, the induced action of $\theta_2^{i+1, 1}$ is given by $\theta_2^{i+1, 1}(g_a^i-g_b^i)=(c_a-c_b)g^i\in H_\emptyset^i(\Delta)\cong H^{i+1}_\mideal(R^1)_1$. That is, $\theta_2^{i+1, 1}$ is an isomorphism as long as $c_a\not= c_b$ (note also that the singularities of $\Delta$ are generically isolated); this means that $H^i_\mideal(R^2)_1=0$. In summary:
\begin{equation}
H^i_\mideal(R^2)_j \cong \left\{\begin{array}{cc} H_\emptyset^{i+1}(\Delta) & j=2 \\ H_\emptyset^i(\Delta) & j= 0 \\ 0 & \text{otherwise.} \end{array}\right.
\end{equation}
Since $H_\mideal^i(R^2)_1=0$, is is immediate that $H_\mideal^i(R_2)$ is quasi-Buchsbaum for all $i$. The specific choice of $\theta_1$ was made for the ease of calculation. For sufficiently generic choices of $\theta_1$ and $\theta_2$, the same isomorphisms hold. In particular, it is evident that the singularities of $\Delta$ are generically isolated.

\section{Results}\label{results}

\subsection{Buchsbaumness}\label{Buchsbaumness}

We now move on to showing whether or not certain modules are Buchsbaum. For this, the following theorem (\cite[Theorem I.3.7]{StVo}) is vital.
\begin{theorem}\label{surjective} Let $\field$ be an infinite field, with $M$ a Noetherian graded $A$-module and $d:=\dim M>0$. Then $M$ is a Buchsbaum module if and only if the natural maps $\varphi_M^i:\Ext_A^i(\field, M)\to H_\mideal^i(M)$ are surjective for $i<d$.
\end{theorem}

Thus far we know some limited information about $\Ext_A^i(\field, R^j)$ and $H_\mideal^i(R^j)$ in terms of the simplicial cohomology of subcomplexes of $\Delta$. Thankfully, Miyazaki has furthered this correspondence with an explicit description of $\varphi_R^i$ in \cite[Corollary 4.5]{canonicalMap}.

\begin{theorem}\label{canonicalMap}The canonical map $\varphi_R^i:\Ext_A^i(\field, R)\to H_\mideal^i(R)$ corresponds to the identity map on $H_\emptyset^{i-1}(\Delta)$ in degree zero and to the direct sum of maps
	\[\bigoplus_{v\in V}\left(\varphi_v^i:H_\emptyset^{i-2}(\cost_\Delta v) \to H_\emptyset^{i-2}(\lk_\Delta v)\right)
	\]
induced by the inclusions in degree $-1$.
\end{theorem}
For our purposes, an alternate expression for $\varphi_R^i$ turns out to be even more powerful than the one above.
For some fixed $v$, consider the long exact sequence in simplicial cohomology for the triple $(\Delta, \cost_\Delta v, \emptyset)$. In our notation, it is written as
\begin{equation}\label{triple}
\cdots \to H_\emptyset^{i-2}(\Delta)\to H_\emptyset^{i-2}(\cost_\Delta v) \xrightarrow{\delta} H_v^{i-1}(\Delta)\xrightarrow{\iota_v^{i-1}} H_\emptyset^{i-1}(\Delta)\to\cdots.
\end{equation}
Under the isomorphism $H_v^{i-1}(\Delta)\cong H_\emptyset^{i-2}(\lk_\Delta v)$, a quick check shows that the connecting homomorphism $\delta$ in this sequence is equivalent to $\varphi_v^i$ in the theorem above (this is also made apparent in examining its proof). On the other hand, if we consider the cohomology modules above as components of $H_\mideal^{i}(R)$ as in Theorem \ref{Grabe}, then the $\iota_v^{i-1}$ map in this sequence is the same as the ``multiplication by $x_v$'' map on $H^i_\mideal(R)_v$. These equivalences along with the exactness of (\ref{triple}) allow us to deduce the following proposition.
\begin{proposition}\label{canonicalMapImage}
	The image of the $H_\emptyset^{i-2}(\cost_\Delta v)$ component of $\Ext_A^i(\field, R)_{-1}$ under the canonical map $\varphi_R^i:\Ext_A^i(\field, R)\to H_\mideal^i(R)$ is precisely the kernel of $\iota^{i-1}_v:H_v^{i-1}(\Delta)\to H_\emptyset^{i-1}(\Delta)$. In particular,
	\[
	(\im\varphi_R^i)_{-1} = K^i
	\]
	through the identifications
	\[
	\left(\im\varphi_R^i\right)_{-1} = \bigoplus_{v\in V}\left(\im\varphi_v^i\right) = \bigoplus_{v\in V} \left(\Ker\iota_v^{i-1}\right) = K^i.
	\]
\end{proposition}
Note that if $\theta$ is any linear form then the proposition immediately implies that \linebreak $(\im \varphi_R^i)_{-1}\subseteq \Ker\theta^{i, 0}$. Examining this containment more closely provides a characterization of the Buchsbaumness of $R^1$. Before stating this characterization, we note that one commutative diagram in particular will be used repeatedly in proving many of our results. Here we explain its origin.

\begin{construction}\label{diagramConstruction}
The short exact sequence (\ref{S1}) induces the following commutative diagram of vector spaces with exact rows:
\[\adjustbox{scale=.90}{%
	\begin{tikzcd}
	\Ext_A^i(\field, R)_{j-1} \arrow{r}{\theta_1^{i, j}}\arrow{d}{\varphi_R^i}& \Ext_A^i(\field, R)_j \arrow{r}{\pi}\arrow{d}{\varphi_R^i}& \Ext_A^i(\field, R^1)_j \arrow{r}{\delta}\arrow{d}{\varphi_{R^1}^i}& \Ext_A^{i+1}(\field, R)_{j-1} \arrow{r}{\theta_1^{i+1, j}}\arrow{d}{\varphi_R^{i+1}}& \Ext_A^{i+1}(\field, R)_j \arrow{d}{\varphi_R^{i+1}}\\
	H_\mideal^i(R)_{j-1} \arrow{r}{\theta_1^{i, j}}& H_\mideal^i(R)_j\arrow{r}{\pi} & H_\mideal^i(R^1)_j \arrow{r}{\delta}& H_\mideal^{i+1}(R)_{j-1} \arrow{r}{\theta_1^{i+1, j}}& H_\mideal^i(R)_j.
\end{tikzcd}
}\]
Since $\theta_1$ is the beginning of a regular sequence for $R$, it acts trivially on $\Ext_A^i(A/\mideal^l, R)$ for all $i$ and $j$ (see, e.g., \cite[p.~272]{HerzogHibi}). By the commutativity of the rightmost square, the image of $\Ext_A^{i+1}(\field, R)_{j-1}$ under $\varphi_R^{i+1}$ must lie in $\Ker\theta_1^{i+1, j}$. On the other hand, the exactness of the bottom row tells us that $\pi:H_\mideal^i(R)_j\to H_\mideal^i(R^1)_j$ factors through the projection $H_\mideal^i(R)_j\to \Coker \theta_1^{i, j}$. As this does not alter the commutativity of the diagram, we can now alter it so that the top and bottom rows are both short exact sequences as follows
\[
	\begin{tikzcd}
	0 \arrow{r}& \Ext_A^i(\field, R)_j \arrow{r}\arrow{d}{}& \Ext_A^i(\field, R^1)_j \arrow{r}\arrow{d}{\varphi_{R^1}^i}& \Ext_A^{i+1}(\field, R)_{j-1} \arrow{r}\arrow{d}{\varphi_R^{i+1}}& 0\\
	0 \arrow{r}& \Coker\theta_1^{i, j}\arrow{r} & H_\mideal^i(R^1)_j \arrow{r}& \Ker\theta_1^{i+1, j} \arrow{r}& 0,
	\end{tikzcd}\]
where the left vertical map is the composition of $\varphi_R^i$ with the projection $H_\mideal^i(R)_j\to\Coker\theta_1^{i, j}$. We can repeat this construction starting with the short exact sequence (\ref{S2}), yielding the same diagram as above with $R$, $R^1$, and $\theta_1$ replaced by $R^1$, $R^2$, and $\theta_2$, respectively.
\end{construction}
Our first use of this construction will be in proving the following proposition (an alternate proof of the ``if'' direction also appears in \cite[Lemma 4.3(2)]{NS-sing}).

\begin{proposition}\label{homIsoIsBuchs}
	If $\Delta$ has isolated singularities, then $R^1$ is Buchsbaum if and only if the singularities of $\Delta$ are homologically isolated.
\end{proposition}

\begin{proof}
	Construction \ref{diagramConstruction} provides the following diagram:
	\[
	\begin{tikzcd}
	0 \arrow{r}& \Ext_A^i(\field, R)_0 \arrow{r}\arrow{d}& \Ext_A^i(\field, R^1)_0 \arrow{r}\arrow{d}{\varphi_{R^1}^i}& \Ext_A^{i+1}(\field, R)_{-1} \arrow{r}\arrow{d}{\varphi_R^{i+1}}& 0 \\
	0 \arrow{r}& \Coker \theta_1^{i, 0}\arrow{r} & H_\mideal^i(R^1)_0 \arrow{r}& \Ker \theta_1^{i+1, 0}\arrow{r} & 0.
	\end{tikzcd}
	\]
	By definition, if $\Delta$ contains singularities that are not homologically isolated then there exists some $i$ such that $K^{i+1}\subsetneq \Ker\theta_1^{i+1, 0}$. By Proposition \ref{canonicalMapImage}, this implies that the $\varphi_R^{i+1}$ map appearing in the diagram above is not a surjection. Since $\varphi_R^i:\Ext_A^i(\field, R)_0\to H_\mideal^i(R)_0$ is an isomorphism, the left vertical map is always a surjection. Then the snake lemma applied to this diagram shows that $\varphi_{R^1}^i$ is not a surjection, so $R^1$ is not Buchsbaum by Theorem \ref{surjective}.
	
	Conversely, if the singularities of $\Delta$ are homologically isolated, then $\Ker\theta_1^{i+1, 0}=K^{i+1}$ for all $i$, so that the $\varphi_R^{i+1}$ map in the diagram is always a surjection. The snake lemma now shows that $\varphi^i_{R^1}$ is a surjection in degree $0$. In degree $1$, we only need to raise the degrees in the previous diagram by one. The diagram simplifies to
		\[
		\begin{tikzcd}
		0 \arrow{r}&  \Ext_A^i(\field, R^1)_1 \arrow{r}\arrow{d}{\varphi_{R^1}^i}& \Ext_A^{i+1}(\field, R)_{0} \arrow{r}\arrow{d}{\varphi_R^{i+1}}& 0 \\
		0 \arrow{r}& H_\mideal^i(R^1)_1 \arrow{r}& H_\mideal^{i+1}(R)_0 \arrow{r}& 0,
		\end{tikzcd}
		\]
	because $\Ext_A^i(\field, R)_1=\Coker \theta_1^{i, 1}=0$. Since $\varphi_R^{i+1}$ is an isomorphism is degree $0$, this completes the proof.
\end{proof}

We have now seen that spaces with homologically isolated singularities are ``close'' to being Buchsbaum in that $R^1$ is Buchsbaum. It is natural to ask whether descending to $R^2$ could always provide a Buchsbaum module, even in the non-homologically-isolated case. This is not true, as exhibited by the following proposition.

\begin{proposition}\label{AlmostBuchsExamples}Suppose $\Delta$ is a space with non-homologically-isolated singularities and that there exists $i$ such that $H_\emptyset^{i-1}(\Delta)=0$, while $\Ker \theta_1^{i+1, 0}\not=0$ and $\iota_v^i$ is injective for all $v$. Then $R^2$ is not Buchsbaum.
\end{proposition}

\begin{proof}The hypotheses combined with the exact sequence in (\ref{triple}) show that $H_\emptyset^{i-1}(\cost v)=0$, so that $\Ext_A^{i+1}(\field, R)_{-1}=0$. Also, $\Coker \theta_1^{i, 0}=0$ because $H_\emptyset^{i-1}(\Delta)=0$. Then the diagram of Construction \ref{diagramConstruction} can be filled in as follows
	\[
	\begin{tikzcd}
	0 \arrow{r}& \Ext_A^i(\field, R)_0 \arrow{r}\arrow{d}& \Ext_A^i(\field, R^1)_0 \arrow{r}\arrow{d}{\varphi_{R^1}^i}& 0 \arrow{r}\arrow{d}& 0 \\
	0 \arrow{r}& 0\arrow{r} & H_\mideal^i(R^1)_0 \arrow{r}& \Ker \theta_1^{i+1, 0}\arrow{r} & 0.
	\end{tikzcd}
	\]
	This demonstrates that $\varphi_{R^1}^i$ is the zero map in degree $0$. Now repeat this argument using $R^1$ and $R^2$ instead of $R$ and $R^1$. In this case, $\Ext_A^{i+1}(\field, R^1)_{-1}=0$ because $\Ext_A^{i+1}(\field, R)_{-1}=0$. Since $H_\mideal^i(R^1)_{-1}=0$, the diagram provided is
	\[
	\begin{tikzcd}
	\Ext_A^i(\field, R^1)_0 \arrow{r}{\sim}\arrow{d}{\varphi_{R^1}^i}& \Ext_A^i(\field, R^2)_0 \arrow{d}{\varphi_{R^2}^i} \\
	 H_\mideal^i(R^1)_0\arrow{r}{\sim} &  H_\mideal^i(R^2)_0,
	\end{tikzcd}
	\]
	showing that $\varphi_{R^2}^i$ is the zero map in degree $0$ as well. Since $H_\mideal^i(R^2)_0\not=0$, Theorem \ref{surjective} completes the proof.
\end{proof}

The hypotheses of this proposition may seem fairly restrictive, but that is not necessarily the case. In fact, choosing $i$ to be the least $i$ such that $H_\emptyset^i(\Delta)\not=0$ when $\Delta$ is a triangulation of the suspension of a manifold that is not a homology sphere will always do the trick, yielding the following Corollary:

\begin{corollary}\label{SuspensionsAlmostBuchs}If $\Delta$ is the suspension of a Buchsbaum complex that is not Cohen-Macaulay, then $R^2$ is not Buchsbaum.
\end{corollary}

\subsection{Almost Buchsbaumness}

Although these $R^2$ modules are not guaranteed to be Buchsbaum when $\Delta$ has non-homologically-isolated singularities, they are ``close'' to being Buchsbaum in some interesting ways and share some of the same properties. The examples above in which $R^2$ is not Buchsbaum fail the criterion of Theorem \ref{surjective} in the degree $0$ piece of $H_\mideal^i(R^2)$. Theorem \ref{surjectivityThm} asserts that (in the generically isolated case) this is the only possible obstruction to satisfying Theorem \ref{surjective}, and we now present its proof.

\begin{proof}[Proof of Theorem \ref{surjectivityThm}] : By the calculations in Section \ref{calculations}, we only need to verify surjectivity in degrees $1$ and $2$. The last diagram in the proof of Proposition \ref{homIsoIsBuchs} holds regardless of whether or not the singularities of $\Delta$ are homologically isolated. Hence, $\varphi_{R^1}^i$ is an isomorphism in degree $1$. Construction \ref{diagramConstruction} then induces the diagram 
		\[
		\begin{tikzcd}
		0 \arrow{r}&  \Ext_A^i(\field, R^2)_2 \arrow{r}\arrow{d}{\varphi_{R^2}^i}& \Ext_A^{i+1}(\field, R^1)_{1} \arrow{r}\arrow{d}{\varphi_{R^1}^{i+1}}& 0 \\
		0 \arrow{r}& H_\mideal^i(R^2)_2 \arrow{r}& H_\mideal^{i+1}(R^1)_1 \arrow{r}& 0,
		\end{tikzcd}
		\]
so that $\varphi_{R^2}^i$ is always an isomorphism in degree $2$.

It remains to show that $\varphi_{R^2}^i$ is a surjection in degree $1$. As usual, we have the following diagram.
\[
	\begin{tikzcd}
		0 \arrow{r}& \Ext_A^{i-1}(\field, R^1)_1 \arrow{r}\arrow{d}& \Ext_A^{i-1}(\field, R^2)_1 \arrow{r}\arrow{d}{\varphi_{R^2}^{i-1}}& \Ext_A^{i}(\field, R^1)_{0} \arrow{r}\arrow{d}& 0 \\
		0 \arrow{r}& \Coker \theta_2^{i-1, 1}\arrow{r} & H_\mideal^{i-1}(R^2)_1 \arrow{r}& \Ker \theta_2^{i, 1}\arrow{r} & 0.
	\end{tikzcd}
\]
According to the previous paragraph, the left vertical map must be a surjection. If we can show that the right vertical map is a surjection as well, then the proof will be complete. The right map is obtained by restricting the range of $\varphi_{R^1}^{i}$ to the subspace $\Ker\theta_2^{i, 1}$ of $H_\mideal^{i}(R^1)_0$. Note that the failure of $\varphi_{R^1}^{i}$ to be a surjection in this degree is precisely what made $R^1$ fail to be Buchsbaum in the non-homologically-isolated case.

Now consider a larger commutative diagram, all of whose rows are exact. All vertical maps are those induced by the action of $\theta_2$, and all maps from the back ``panel'' to the front are induced by the canonical maps $\varphi_{R^j}^i$.
\[
\adjustbox{scale=.65}{%
	\begin{tikzcd}
		&
		0
		\ar{rr}
		& & \Ext_A^i(\field, R)_0
		\ar{rr}
		\ar{dd}{}
		\ar{dl}[swap, sloped, near start]{}
		& & \Ext_A^i(\field, R^1)_0
		\ar{rr}
		\ar{dd}{}
		\ar{dl}[swap, sloped, near start]{}
		& & \Ext_A^{i+1}(\field, R)_{-1}
		\ar{rr}
		\ar{dd}{}
		\ar{dl}[swap, sloped, near start]{}
		& & 0
		\\
		0
		\ar[crossing over]{rr}[near start]{}
		& & \Coker\theta_1^{i, 0}
		\ar[crossing over]{rr}
		& & H_\mideal^i(R^1)_0
		\ar[crossing over]{rr}
		& & \Ker \theta_1^{i+1, 0}
		\ar[crossing over]{rr}
		& & 0
		\\
		&
		0
		\ar[near start]{rr}{}
		& & \Ext_A^i(\field, R)_1
		\ar{rr}
		\ar{dl}
		& & \Ext_A^i(\field, R^1)_1
		\ar{rr}
		\ar{dl}
		& & \Ext_A^{i+1}(\field, R)_{0}
		\ar{rr}
		\ar{dl}
		& & 0
		\\
		0
		\ar{rr}
		& & 0
		\ar{rr}
		\ar[crossing over, leftarrow, near start]{uu}{}
		& & H_\mideal^i(R^1)_1
		\ar{rr}
		\ar[crossing over, leftarrow, near start]{uu}{}
		& & H_\mideal^{i+1}(R)_0
		\ar{rr}
		\ar[crossing over, leftarrow, near start]{uu}{}
		& & 0.
	\end{tikzcd}
}\]
Now consider applying the snake lemma to both the front panel and the back panel. Note that the vertical maps on the back panel are all identically zero, since $\theta_2$ acts trivially on all of the modules there. Denote by $\tau$ the restriction of $\theta_2^{i+1, 0}$ to $\Ker\theta_1^{i+1, 0}$, appearing as the right vertical map in the front panel of the diagram. By the naturality of the sequence induced by the snake lemma, we obtain maps from the ``top" panel as follows:
\[
\begin{tikzcd}
0 \ar{r} & \Ext_A^i(\field, R)_0 \ar{r} \ar{d}& \Ext_A^i(\field, R^1)_0 \ar{r} \ar{d}& \Ext_A^{i+1}(\field, R)_{-1} \ar{r} \ar{d}& 0\\
0 \ar{r} & \Coker \theta_1^{i, 0} \ar{r} & \Ker\theta_2^{i, 0} \ar{r} & \Ker \tau \ar{r} & 0.
\end{tikzcd}
\]
Since the left vertical map is a surjection, we will be done if we can show that the right vertical map is a surjection. However, $\Ker\tau$ is simply
\[
(\Ker \theta_1^{i+1, 0})\cap (\Ker \theta_2^{i+1, 2}):= L^{i+1}.
\]
Note that the singularities of $\Delta$ are generically isolated if and only if $L^{i+1}=K^{i+1}$. On the other hand, $K^{i+1}$ is precisely the image of $\Ext_A^i(\field, R)_{-1}$ under $\varphi_R^i$, completing the proof.
\end{proof}

The intersection $L^{i+1}$ above is also central to the proof of Theorem \ref{annihilateThm}, which we now present as well.

\begin{proof}[Proof of Theorem \ref{annihilateThm}]
Once more, since $H_\mideal^i(R^2)$ may only have non-zero components in the graded degrees $0$, $1$, and $2$, we only need to check that $\mideal$ acts trivially on the degree $0$ and degree $1$ parts. We begin in degree $1$. Let $\alpha_v$, $\beta_v$, and $\gamma_v$ all denote the map induced by multiplication by $x_v$ in the diagram below.
\[
\begin{tikzcd}
0 \arrow{r} & \Coker\theta_2^{i, 1} \arrow{r} \arrow{d}{\alpha_v} & H_\mideal^i(R^2)_1 \arrow{r} \arrow{d}{\beta_v} & \Ker\theta_2^{i+1, 1} \arrow{r} \arrow{d}{\gamma_v}& 0 \\
0 \arrow{r} & 0 \arrow{r} & H_\mideal^i(R^2)_2 \arrow{r} & H_\mideal^{i+1}(R^1)_1 \arrow{r} & 0.
\end{tikzcd}
\]
The snake lemma provides the exact sequence
\[
0\to \Coker\theta_2^{i, 1} \to\Ker\beta_v\to\Ker\gamma_v\to0.
\]
Comparing this to the top row of the previous diagram, if we can show that $\Ker\gamma_v$ is all of $\Ker\theta_2^{i+1, 1}$, then we may conclude that $\Ker\beta_v$ is all of $H_\mideal^i(R^2)_1$. Note that $\Ker\theta_2^{i+1, 1}$ is a submodule of $H_\mideal^{i+1}(R^1)_0$; to study this submodule, consider the following diagram with exact rows.
\[
\begin{tikzcd}
0 \arrow{r} & \Coker\theta_1^{i+1, 0} \arrow{r} \arrow{d}{\theta_2^{i+1, 1}}& H_\mideal^{i+1}(R^1)_0 \arrow{r}\arrow{d}{\theta_2^{i+1, 1}} & \Ker\theta_1^{i+2, 0} \arrow{r} \arrow{d}{\tau} & 0\\
0 \arrow{r} & 0 \arrow{r} & H_\mideal^{i+1}(R^1)_1 \arrow{r} & H_\mideal^{i+2}(R)_0 \arrow{r} & 0.
\end{tikzcd}
\]
As in the previous proof, the rightmost vertical map $\tau$ is the restriction of $\theta_2^{i+2, 0}$ to $\Ker\theta_1^{i+2, 0}$. Once more, note that
\[
\Ker \tau = (\Ker \theta_1^{i+2, 0})\cap (\Ker \theta_2^{i+2, 2}):=L^{i+2}.
\]

Through another application of the snake lemma, we get a short exact sequence fitting into the top row of the following diagram
\[
\begin{tikzcd}
0\arrow{r} & \Coker\theta_1^{i+1, 0}\arrow{r}\arrow{d}{\cdot x_v}& \Ker\theta_2^{i+1, 1}\arrow{r}\arrow{d}{\cdot x_v}& L^{i+2} \arrow{r} \arrow{d}{\cdot x_v}&0\\
0 \arrow{r} & 0 \arrow{r} & H_\mideal^{i+1}(R^1)_1 \arrow{r} & H_\mideal^{i+2}(R)_0 \arrow{r} & 0.
\end{tikzcd}
\]
So, it now remains to show that the rightmost map is zero. But $x_v$ acts trivially on $L^{i+2}$ for all $i$ and for all $v$ if and only if
\[
L^{i+2} = K^{i+2},
\]
i.e., if and only if the singularities of $\Delta$ are generically isolated.

Now we show that $\mideal\cdot H_\mideal^i(R^2)_0=0$, independent of the type of isolation of the singularities of $\Delta$. Consider the diagram below:
\[
\begin{tikzcd}
	0\arrow{r} & H_\mideal^i(R^1)_0\arrow{r}\arrow{d}{\cdot x_v}& H_\mideal^i(R^2)_0\arrow{r}\arrow{d}{\cdot x_v}&0\\
	H_\mideal^i(R^1)_0\arrow{r}{\theta_2^{i, 1}} & H_\mideal^i(R^1)_1\arrow{r}& H_\mideal^i(R^2)_1\arrow{r}&H_\mideal^{i+1}(R^1)_0.
\end{tikzcd}
\]
From the exactness of the rows of this diagram, we can conclude that $x_v\cdot H_\mideal^i(R^2)_0=0$ if
\[
x_v\cdot H_\mideal^i(R^1)_0\subseteq \theta_2\cdot H_\mideal^i(R^1)_0.
\]
Generically, all coefficients of $\theta_2$ are non-zero. Combining this with the structure of $H_\mideal^i(R)$ outlined in Theorem \ref{Grabe}, it is immediate that
\[
x_v\cdot H_\mideal^i(R)_{-1}\subseteq \theta_2\cdot H_\mideal^i(R)_{-1}.
\]
The following diagram now completes the proof.
\[
\begin{tikzcd} 0 \arrow{r} & H_\mideal^i(R^1)_1 \arrow{r}{\delta} & H_\mideal^{i+1}(R)_0 \arrow{r} & 0\\
H_\mideal^i(R)_{0} \arrow{r} & H_\mideal^i(R^1)_0 \arrow{r}{\delta}  \arrow{u}{\cdot x_v}\arrow{d}{\cdot\theta_2}& H_\mideal^{i+1}(R)_{-1}\arrow{r}\arrow{u}{\cdot x_v} \arrow{d}{\cdot\theta_2}& H_\mideal^{i+1}(R)_0\\
0 \arrow{r} & H_\mideal^i(R^1)_1 \arrow{r}{\delta}  & H_\mideal^{i+1}(R)_0 \arrow{r} & 0
\end{tikzcd}
\]
\end{proof}

\begin{example}
Combining Theorem \ref{annihilateThm} with Corollary \ref{SuspensionsAlmostBuchs} provides an infinite family of interesting examples of rings with some prescribed properties that are quasi-Buchsbaum but not Buchsbaum. In particular, let $M$ be a $d$-dimensional manifold that is not a homology sphere and let $\Gamma$ be an arbitrary triangulation of $M$. Suppose further that
\[
\max\{i: \Gamma^{(i)}\text{ is Cohen-Macaulay}\}=r+1.
\]
Equivalently, the depth of $\field[\Gamma]$ is $r+1$. If we set $\Gamma'$ to be the join of $\Gamma$ with two points, then $\Gamma'$ is a triangulation of the suspension of $M$ and the depth of $\field[\Gamma']$ is $r+2$. Since the depth of the Stanley--Reisner ring is a topological invariant (\cite[Theorem 3.1]{Munkres}), if $\Delta$ is an arbitrary triangulation of the suspension of $M$ then $R^2$ has depth $r$. Since the singularities of $\Delta$ are generically isolated, $R^2$ is a quasi-Buchsbaum ring of Krull dimension $d$ and depth $r$ that is not Buchsbaum; furthermore, the canonical maps $\varphi_{R^2}^i:\Ext_A^i(\field, R^2)\to H_\mideal^i(R^2)$ are surjections in all degrees except $0$.
\end{example}

\subsection{Another surjection}

By \cite[Proposition I.3.4]{StVo}, the quasi-Buchsbaum property of some $R^2$ is equivalent to the fact that every homogeneous system of parameters of $R^2$ contained in $\mideal^2$ is a weakly regular sequence. In light of the typical definition of the Buchsbaum property in terms of l.s.o.p.'s being weakly regular sequences (see \cite{StVo})) along with the characterization appearing in Theorem \ref{surjective} by surjectivity of the maps $\varphi_M^i:\Ext_A^i(\field, M)\to H_\mideal^i(M)$, what happens when we consider the natural maps $\Ext_A^i(A/\mideal^2, R^2)$ instead? Our next result establishes another measure of the gap between the structure of $R^2$ and the Buchsbaum property:

\begin{proposition}Suppose $\Delta$ has isolated singularities. Then the canonical maps $\psi_{R^2}^i:\Ext_A^i(A/\mideal^2, R^2)\to H_\mideal^i(R^2)$ are surjective.
\end{proposition}

\begin{proof}
Once more, surjectivity needs only to be demonstrated in degrees $0$, $1$, and $2$. We will begin with the degree $2$ piece. The exact sequence (\ref{S1}) with $j=1$ and $l=2$ gives rise to the commutative diagram below, where the horizontal maps are isomorphisms.
\[
\begin{tikzcd}
\Ext^i_A(A/\mideal^2,R^1)_1\arrow{r}{\delta}\arrow{d}{\psi_{R^1}^i} & \Ext^{i+1}_A(A/\mideal^2,R)_0 \arrow{d}{\psi_{R}^{i+1}}\\
H^i_\mideal(R^1)_1 \arrow{r}{\delta} & H^{i+1}_\mideal(R)_0
\end{tikzcd}
\]
By \cite[Corollary 4.5]{canonicalMap}, the map $\psi_R^{i+1}:\Ext_A^{i+1}(A/\mideal^2, R)_0\to H^{i+1}_\mideal(R)_0$ is equivalent to the identity map on $H_\emptyset^{i}(\Delta)$. Hence, $\psi_{R^1}^{i+1}$ is an isomorphism in degree $1$. By the same argument, the exact sequence (\ref{S2}) and the corresponding commutative diagram show that $\psi_{R^2}^i:\Ext^i_A(A/\mideal^2, R^2)_2\to H^i_\mideal(R^2)_2$ is an isomorphism.

The arguments for further graded pieces are similar. First consider the following diagram with exact rows again induced by (\ref{S1}).
\[
\begin{tikzcd} 0 \arrow{r} & \Ext_A^i(A/\mideal^2, R)_0 \arrow{d}[swap]{\psi_R^i} \arrow{r} & \Ext_A^i(A/\mideal^2, R^1)_0 \arrow{d}{\psi_{R^1}^i} \arrow{r} & \Ext_A^{i+1}(A/\mideal^2, R)_{-1} \arrow{d}{\psi_R^{i+1}} \arrow{r} & 0 \\
\, & H_\mideal^i(R)_0 \arrow{r} & H_\mideal^i(R^1)_0 \arrow{r} & H_\mideal^{i+1}(R)_{-1}
\end{tikzcd}
\]
By \cite[Corollary 2]{characterizations}, the left and right vertical maps are isomorphisms. Exactness then implies that the middle vertical map is surjective. So, $\psi_{R^1}^i$ is an isomorphism in degree $1$ and a surjection in degree $0$. For the surjectivity of $\psi_{R^2}^i$ in degree $1$, consider the following commutative diagram induced by (\ref{S2}) with exact rows.
\[
\begin{tikzcd} 0 \arrow{r} & \Ext_A^i(A/\mideal^2, R^1)_1 \arrow{d}[swap]{\psi_{R^1}^i} \arrow{r} & \Ext_A^i(A/\mideal^2, R^2)_1 \arrow{d}{\psi_{R^2}^i} \arrow{r} & \Ext_A^{i+1}(A/\mideal^2, R^1)_{0} \arrow{d}{\psi_{R^1}^{i+1}} \arrow{r} & 0 \\
\, & H_\mideal^i(R^1)_1 \arrow{r} & H_\mideal^i(R^2)_1 \arrow{r} & H_\mideal^{i+1}(R^1)_{0}
\end{tikzcd}
\]
We have just demonstrated that the left and right vertical maps are at least surjections. Again by exactness, we now have that $\psi_{R^2}^i$ is a surjection in degree $1$. Now consider one final commutative diagram.
\[
\begin{tikzcd} \Ext_A^i(A/\mideal^2, R^1)_0 \arrow{r} \arrow{d}{\psi_{R^1}^i} &  \Ext_A^i(A/\mideal^2, R^2)_0 \arrow{d}{\psi_{R^2}^i}\\
H_\mideal^i(R^1)_0 \arrow{r} & H_\mideal^i(R^2)_0
\end{tikzcd}
\]
From Section \ref{calculations}, we know that the bottom map is an isomorphism. Since $\psi_{R^1}^i$ is a surjection in degree zero, $\psi_{R^2}^i$ must be as well.
\end{proof}

\section{An enumerative theorem}\label{enumerative}
Although $R^2$ is not Buchsbaum, the quasi-Buchsbaum property does allow for a computation of the Hilbert series of the generic Artinian reduction of $R$ by a h.s.o.p. of a particular type. Let $\Delta$ have generically isolated singularities and say $\Theta=\theta_1, \ldots, \theta_d$ is a homogeneous system of parameters for $\Delta$ such that $\theta_1, \theta_2$ is a linear regular sequence, while $\theta_3, \ldots, \theta_d$ are quadratic forms.
For $2\le i\le d-1$, there are exact sequences
\[
0\to (0:_{R^i}\theta_{i+1})_{j-2} \to R^i_{j-2}\xrightarrow{\cdot\theta_{i+1}} R^i_j\to R^{i+1}\to 0.
\]
Since $R^2$ is quasi-Buchsbaum, the sequence $\theta_3, \ldots, \theta_d$ is a weakly regular sequence by \cite[Proposition I.2.1(ii)]{StVo}. Furthermore, the proof of the proposition shows that $(0:_{R^2}\theta_3)=H_\mideal^0(R^2)$. On the other hand, \cite[Theorem 3.6]{S-Quasi} states that $R^i$ is quasi-Buchsbaum for $2\le i\le d-1$. Hence, 
the sequence above can be re-written as 
\[
0\to H_\mideal^0(R^i)_{j-2}\to  R^i_{j-2}\to R^i_j\to R^{i+1}\to 0.
\]
for $2\le i\le d-1$. If $\Hilb(M; t)$ denotes the Hilbert series of a $\ZZ$-graded $A$-module $M$, then these exact sequences imply
\[
\Hilb(R^{i+1}; t)=(1-t^2)\Hilb(R^i; t)+t^2\Hilb(H_\mideal^0(R^i); t).
\]
A standard calculation then shows
\begin{equation}\label{HilbertEquality}
\Hilb(R^d; t)=(1+t)^{d-2}(1-t)^d\Hilb(R; t)+\sum_{i=2}^{d-1}\left[t^2(1-t^2)^{d-1-i}\Hilb(H_\mideal^0(R^i); t)\right].
\end{equation}
The first term reduces to $(1+t)^{d-2}\sum_{i=0}^d h_i(\Delta)t^i$, following \cite{St-UBC}. To analyze the sum, \cite[Lemma 3.5]{S-Quasi} provides the exact sequence
\[
0\to H_\mideal^j(R^i)_k\to H_\mideal^j(R^{i+1})_k\to H_\mideal^{j+1}(R^i)_{k-2}\to 0
\]
for $2\le i\le d-2$ and $0\le j\le d-i-2$. So, as vectors spaces over $\field$, for $2\le i\le d-1$ there are isomorphisms
\[
H_\mideal^0(R^i)_j\cong\bigoplus_{j=0}^{i-2}\left(\bigoplus_{i-2\choose j} H_\mideal^j(R^2)_{-2j}\right).
\]
That is,
\begin{equation}\label{hilbEq1}
\Hilb(H_\mideal^0(R^i); t)=\sum_{j=0}^{i-2}{i-2\choose j}t^{2j}\Hilb(H_\mideal^j(R^2); t).
\end{equation}
Now define
\[
\mu^i= \dim_\field H_\mideal^i(R^2)_0=\dim_\field \left(\Coker\theta_1^{i, 0}\oplus \Ker\theta_1^{i+1, 0}\right),
\]
\[
\nu^i= \dim_\field H_\mideal^i(R^2)_1=\dim_\field \left(\Coker\theta_2^{i, 1}\oplus \Ker\theta_2^{i+1, 1}\right),
\]
and
\[
\beta_\emptyset^{i}(\Delta)= \dim_\field H_\emptyset^i(\Delta),
\]
so
\[
\Hilb(H_\mideal^i(R^2); t)=\mu^i+\nu^it+\beta_\emptyset^{i+1}(\Delta)t^2.
\]
Combining this equality with equations (\ref{HilbertEquality}) and (\ref{hilbEq1}) implies the following theorem describing $\Hilb(R^d; t)$.
\begin{theorem}\label{enumerativeThm}
If $q=2p$ is even, then 
\[
\dim_\field(R^d_q)=\sum_{i=0}^q{d-2\choose q-i}h_i(\Delta)+(-1)^{p-1}{d-2\choose p}\sum_{k=0}^{p-1}\left[(-1)^k\left(\mu^k+\frac{p\beta_\emptyset^k(\Delta)}{d-1-p}\right)\right].
\]
If $q=2p+1$ is odd, then
\[
\dim_\field(R^d_q)=\sum_{i=0}^q{d-2\choose q-i}h_i(\Delta)+(-1)^{p-1}{d-2\choose p}\sum_{k=0}^{p-1}(-1)^k\nu^k.
\]

\end{theorem}

\begin{remark}
	Note that $\mu^i$ is actually a topological invariant of $\Delta$. If $\|\Delta\|$ is the geometric realization of $\Delta$ and $\Sigma$ is the set of isolated singularities of $\Delta$, then $\mu^i = \dim_\field H_\emptyset^{i-1}(\|\Delta\|\smallsetminus \Sigma)$ (see \cite[Theorem 4.7]{NS-sing}). At present there is no similar description for $\nu^i$, as it is not clear how to trace the geometry of $\Delta$ all the way through to $H_\mideal^i(R^2)_1$ in such a precise manner.
\end{remark}

\section{Comments}\label{comments}
There are many possible abstractions of the results that have been presented. Perhaps the most immediate consideration is in introducing singularities of dimension greater than $0$. In this case, the structure of $H_\mideal^i(\field[\Delta])$ outlined in Theorem \ref{Grabe} becomes more involved and hinders the calculations of Section \ref{calculations}. For instance, if $\Delta$ contains singular faces even of dimension $1$, then $\theta_1^{i, j}$ will not, in general, be an isomorphism in degrees $j<0$. This implies that there is some $i$ such that $H_\mideal^i(\field[\Delta]/\theta_1\field[\Delta])_j$ is non-zero for infinitely many values of $j$, i.e., $\field[\Delta]/\theta_1\field[\Delta]$ does not have finite local cohomology. In fact, Miller, Novik, and Swartz classified when this is the case for quotients of $\field[\Delta]$ by arbitrarily many generic linear forms:

\begin{theorem}
	(\cite[Theorem 2.4]{MNS-sing}) A simplicial complex $\Delta$ is of singularity dimension at most $m-1$ if and only if $\field[\Delta]/(\theta_1, \ldots, \theta_m)\field[\Delta]$ has finite local cohomology.
\end{theorem}

In the case that $m=1$, we know that $\field[\Delta]/\theta_1\field[\Delta]$ not only has finite local cohomology, but it is also Buchsbaum if and only if the singularities of $\Delta$ are homologically isolated. So, one may pose the following question.

\begin{question}If $\Delta$ is of singularity dimension $m-1$, is there an analog of the homological isolation property for singularities of arbitrary dimension classifying when $\field[\Delta]/(\theta_1, \ldots, \theta_m)\field[\Delta]$ is Buchsbaum?
\end{question}

A possible property could be that all pairs of images of maps of the form $H^i(\Delta, \cost_\Delta (F\cup\{u\}))\to H^i(\Delta, \cost_\Delta F)$ and $H^i(\Delta, \cost_\Delta (F\cup\{v\}))\to H^i(\Delta, \cost_\Delta F)$ occupy linearly independent subspaces of $H^i(\Delta, \cost_\Delta F)$ for all faces $F$ and all vertices $u$ and $v$ in the appropriate dimensions.

On the other hand, when $m=1$ we know that $\field[\Delta]/\theta_1\field[\Delta]$ has finite local cohomology and that $\field[\Delta]/(\theta_1, \theta_2)\field[\Delta]$ is quasi-Buchsbaum if and only if the singularities of $\Delta$ are generically isolated. The leads to our next question.

\begin{question}If $\Delta$ is a simplicial complex of singularity dimension $m-2$ and $\field[\Delta]$ is of depth at least $m$ with $\theta_1, \ldots, \theta_{m}$ a regular sequence on $\field[\Delta]$, is there an analog of the generic isolation property classifying when $\field[\Delta]/(\theta_1, \ldots, \theta_{m})\field[\Delta]$ is quasi-Buchsbaum?
\end{question}

Again, one candidate property would be that given $m+1$ generic linear forms, the pairwise intersections
\[
\Ker \theta_i^{l, 0}\cap\Ker\theta_j^{l, 0}
\]
are all trivially equal to $K^l$.

Lastly, we have been able to provide many examples of complexes $\Delta$ with isolated singularities in which $\field[\Delta]/(\theta_1, \theta_2)\field[\Delta]$ is quasi-Buchsbaum but not Buchsbaum. In light of these examples, we present the following conjecture.

\begin{conjecture}\label{neverBuchsbaum}
	In the setting of Theorem \ref{annihilateThm}, if the singularities of $\Delta$ are not homologically isolated then $\field[\Delta]/(\theta_1, \theta_2)\field[\Delta]$ is never Buchsbaum.
\end{conjecture}

\section*{Acknowledgements}The author would like to thank Isabella Novik for proposing a problem that led to the results in this paper. The author is also grateful to Satoshi Murai for pointing out an error in an earlier version, leading to the notion of generic isolation of singularities.

\bibliographystyle{alpha}
\bibliography{singularities-biblio}

\end{document}